\documentclass[11pt]{amsart}  
\usepackage{amsbsy,amsmath,amsthm,amssymb,color,verbatim,ulem,color}
\usepackage[backrefs]{amsrefs}
\usepackage{float}
\usepackage{bm}
\usepackage{fullpage,cancel,hyperref}
\usepackage{float}
\usepackage{wrapfig}
\usepackage[dvipsnames]{xcolor}
\usepackage{graphicx}
\usepackage{subcaption}
\usepackage{multicol}
\usepackage{multirow}
\usepackage{longtable}
\usepackage{enumitem}

\usepackage{makecell}
\setcellgapes{4pt}
\usepackage{supertabular}
\usepackage{tikz}
\usepackage{tkz-graph}
\tikzstyle{vertex}=[circle, draw, inner sep=0pt, minimum size=4pt]

\tikzstyle{vtx}=[circle, draw, inner sep=0pt, minimum size=8pt]
\usetikzlibrary{decorations.pathreplacing}

\definecolor{darkgreen}{cmyk}{.9,0,.9,.2}
\definecolor{midgray}{gray}{0.60}
\definecolor{lightgray}{gray}{0.90}
\definecolor{lmgray}{gray}{0.70}
\usepackage{colortbl}
\usetikzlibrary{arrows.meta,calc,decorations.markings,math,arrows.meta,patterns}
\usepackage{caption}
\usepackage{subcaption}
\usepackage{mathdots}
\usepackage{yhmath}
\usepackage{cancel}
\usepackage{color}
\usepackage{siunitx}
\usepackage{array}
\usepackage{multirow}
\usepackage{amssymb}
\usepackage{gensymb}
\usepackage{tabularx}
\usepackage{booktabs}
\usetikzlibrary{fadings}

\makeatletter
\def\@tocline#1#2#3#4#5#6#7{\relax
  \ifnum #1>\c@tocdepth 
  \else
    \par \addpenalty\@secpenalty\addvspace{#2}%
    \begingroup \hyphenpenalty\@M
    \@ifempty{#4}{%
      \@tempdima\csname r@tocindent\number#1\endcsname\relax
    }{%
      \@tempdima#4\relax
    }%
    \parindent\z@ \leftskip#3\relax \advance\leftskip\@tempdima\relax
    \rightskip\@pnumwidth plus4em \parfillskip-\@pnumwidth
    #5\leavevmode\hskip-\@tempdima
      \ifcase #1
       \or\or \hskip 1em \or \hskip 2em \else \hskip 3em \fi%
      #6\nobreak\relax
    \hfill\hbox to\@pnumwidth{\@tocpagenum{#7}}\par
    \nobreak
    \endgroup
  \fi}
\makeatother

\newtheorem{problem}{Problem}

\usepackage{relsize}
\makeatletter
\newtheorem*{rep@theorem}{\rep@title}
\newcommand{\newreptheorem}[2]{%
\newenvironment{rep#1}[1]{%
 \def\rep@title{#2 \ref{##1}}%
 \begin{rep@theorem}}%
 {\end{rep@theorem}}}
\makeatother

\makeatletter
\newtheorem*{rep@conjecture}{\rep@title}
\newcommand{\newrepconjecture}[2]{%
\newenvironment{rep#1}[1]{%
 \def\rep@title{#2 \ref{##1}}%
 \begin{rep@conjecture}}%
 {\end{rep@conjecture}}}
\makeatother

\makeatletter
\newcommand{\addresseshere}{%
  \enddoc@text\let\enddoc@text\relax
}
\makeatother

\theoremstyle{definition}
\newtheorem{definition}{Definition}[section]
\newtheorem{theorem}[definition]{Theorem}

\newtheorem{conjecture}[definition]{Conjecture}
\newtheorem{lemma}[definition]{Lemma}
\newtheorem{corollary}[definition]{Corollary}
\newtheorem{example}[definition]{Example}
\newreptheorem{theorem}{Theorem}
\newrepconjecture{conjecture}{Conjecture}

\newcommand{\Z}{\mathbb{Z}}
\newcommand{\N}{\mathbb{N}}
\newcommand{\Tau}{\mathrm{T}}

\newcommand{\containedpf}{$B_{n,k}$ }
\newcommand{\standardpf}{$PF_n$ }

\allowdisplaybreaks

\def\al{\alpha}

\title{A generalization of parking functions\\allowing backward movement}

\author{Alex Christensen}
\address{University of Arizona, Department of Mathematics, United States}
\email{\textcolor{blue}{\href{mailto:ajc333@comcast.net}{ajc333@comcast.net}}}

\author{Pamela E. Harris}
\address{Department of Mathematics and Statistics, Williams College, United States}
\email{\textcolor{blue}{\href{mailto:peh2@williams.edu}{peh2@williams.edu}}}

\author{Zakiya Jones}
\address{Pomona College, Department of Mathematics, United States}
\email{\textcolor{blue}{\href{mailto:zakiyacmjones@gmail.com}{zakiyacmjones@gmail.com}}}

\author{Marissa Loving}
\address{Department of Mathematics, University of Illinois at Urbana-Champaign, United States}
\email{\textcolor{blue}{\href{mailto:mloving2@illinois.edu}{mloving2@illinois.edu}}}

\author{Andr\'es Ramos Rodr\'iguez}
\address{University of Puerto Rico, Rio Piedras, Department of Computer Science}
\email{\textcolor{blue}{\href{mailto:ramosandres443@gmail.com }{ramosandres443@gmail.com}}}

\author{Joseph Rennie}
\address{Department of Mathematics, University of Illinois at Urbana-Champaign, United States}
\email{\textcolor{blue}{\href{mailto:rennie2@illinois.edu}{rennie2@illinois.edu}}}

\author{Gordon Rojas Kirby}
\address{UC Santa Barbara, Department of Mathematics, United States}
\email{\textcolor{blue}{\href{mailto:gkirby@math.ucsb.edu}{gkirby@math.ucsb.edu}}}

\begin{document}

\maketitle

\begin{abstract}
Classical parking functions are defined as the parking preferences for $n$ cars driving (from west to east) down a one-way street containing parking spaces labeled from $1$ to $n$ (from west to east). Cars drive down the street toward their preferred spot and park there if the spot is available. Otherwise, the car continues driving down the street and takes the first available parking space, if such a space exists. If all cars can park using this parking rule, we call the $n$-tuple containing the cars' parking preferences a parking function.

In this paper, we introduce a generalization of the parking rule allowing cars whose preferred space is taken to first proceed up to $k$ spaces west of their preferred spot to park before proceeding east if all of those $k$ spaces are occupied. We call parking preferences which allow all cars to park under this new parking rule $k$-Naples parking functions of length $n$.
This generalization gives a natural interpolation between classical parking functions, the case when $k=0$, and all $n$-tuples of positive integers $1$ to $n$, the case when $k\geq n-1$. Our main result provides a recursive formula for counting $k$-Naples parking functions of length $n$. We also give a characterization for the $k=1$ case by introducing a new function that maps $1$-Naples parking functions to classical parking functions, i.e. $0$-Naples parking functions. Lastly, we present a bijection between $k$-Naples parking functions of length $n$ whose entries are in weakly decreasing order and a family of signature Dyck paths. 
\end{abstract}

\section{Introduction}
Parking functions were introduced independently by Ronald Pyke and by Alan Konheim and Benjamin Weiss in relation to hashing problems \cites{Pyke,KonheimAndWeiss}.
Parking functions are combinatorial objects defined as follows. Let the set of natural numbers be defined as $\N := \{1,2,3, \dots\}$, and for \mbox{$n\in\N$} let $[n]:=\{1,\dots, n\}$. Now, consider $n$ parking spaces on a one-way street arranged in a line numbered $1$ to $n$ from west to east. Suppose there are $n$ cars, denoted $c_1,c_2,\ldots,c_n$, that drive in order down this one-way street. For all $1\leq i\leq n$, each car $c_i$ has a preferred parking spot \mbox{$a_i\in [n]$} and multiple cars are allowed to have the same preference. This is illustrated\footnote{
Black car vector. Digital image. The London Telegraph. 13 August 2019, \url{https://www.goodfreephotos.com/vector-images/black-car-vector.png.php}.}
 in Figure \ref{fig:pfillustration}.

\begin{figure}[h]
    \centering
    \begin{tikzpicture}
    \node[inner sep=0pt](a1) at (-1.5,.6)
    {$c_1$};
    \node[inner sep=0pt](a1) at (-1.5,.2)
    {\includegraphics[width=.4in]{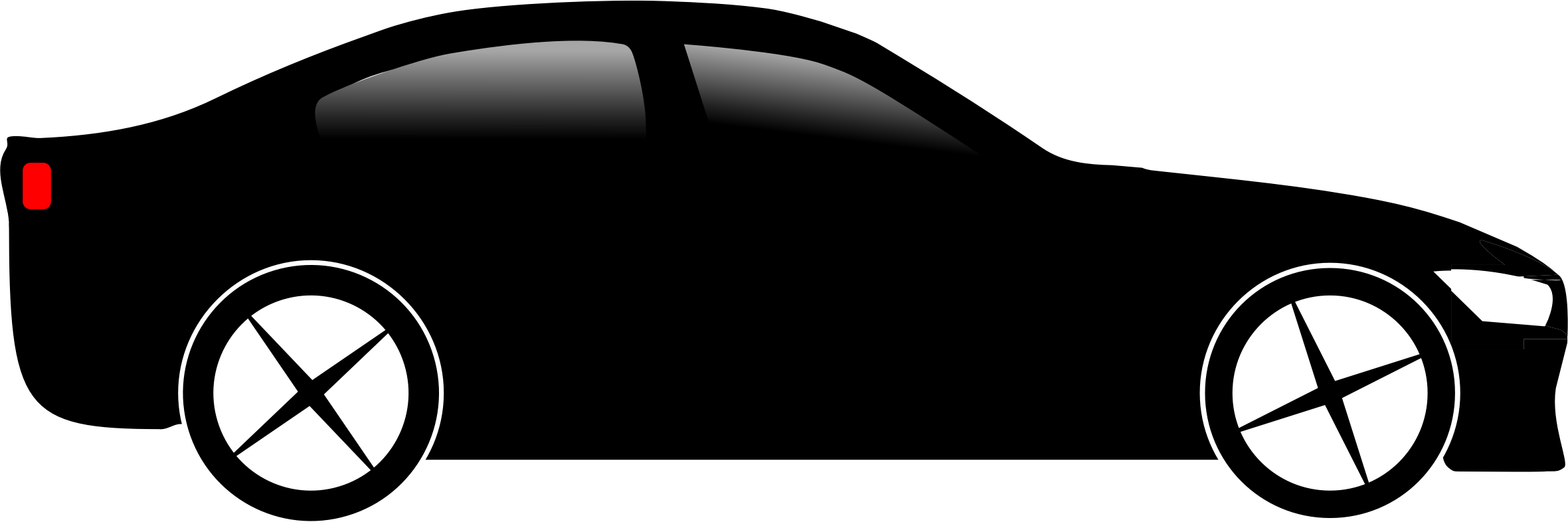}};
    \node[inner sep=0pt](a1) at (-2.75,.6)
    {$c_2$};
    \node[inner sep=0pt](a1) at (-2.75,.2)
    {\includegraphics[width=.4in]{black-car-vector.png}};
    \node[inner sep=0pt](a1) at (-4.75,.6)
    {$c_n$};
    \node[inner sep=0pt](a1) at (-4.75,.2)
    {\includegraphics[width=.4in]{black-car-vector.png}};
    \node[inner sep=0pt](a1) at (-3.75,.2)
    {$\cdots$};
    \node[inner sep=0pt](a1) at (-.5,.2)
    {$\longrightarrow$};
            \draw[thick] (0,0) to node[below]{1} (0.5,0);
    \draw[thick] (1,0) to node[below]{2} (1.5,0);
    \draw[thick] (2,0) to node[below]{3} (2.5,0);
    \node at (3.25,0) {$\dots$};
    \draw[thick] (4,0) to node[below]{$n$} (4.5,0);
    \end{tikzpicture}
    \caption{Parking function illustration.}
    \label{fig:pfillustration}
\end{figure}
A \textbf{parking preference} of length $n$ is an $n$-tuple of integers in $[n]$ where the $i$-th component corresponds to the preferred parking spot of car $c_i$. We denote the set of parking preferences of length $n$ as $PP_n$. Note that $|PP_n|=n^n$. For a parking preference  $\alpha=(a_1,\dots,a_n)\in PP_n$, we establish the following parking rule: for all $1\leq i\leq n$, $c_i$ starts at parking space $1$ and drives toward its preferred parking spot $a_i$. If $a_i$ is unoccupied $c_i$ parks. Otherwise, $c_i$ proceeds forward until it reaches the next available parking spot. If every parking spot numbered from $a_i$ up to and including $n$ is taken, then $c_i$ is unable to park. On the other hand, if every car is able to park given the preference $\al\in PP_n$, then we say that $\alpha$ is a \textbf{parking function}. 
A necessary and sufficient condition to determine if a parking preference  $\alpha = (a_1, a_2, \dots, a_n) \in PP_n$ is a parking function is determined by considering $\beta = (b_1,\dots,b_n)$ which is the increasing rearrangement of the entries in $\alpha$. Then, $\alpha$ is a parking function if and only if $b_i \leq i$ for each $i$. We denote the set of all parking functions of length $n$ as $PF_n$. It is known that $|PF_n|=(n+1)^{n-1}$ (see \cite{KonheimAndWeiss}).

Parking functions are interesting in their own right and have applications in combinatorics, group theory, the study of hyperplane arrangements, and computer science. Many generalizations of parking functions exist and the main results give formulas to count the number of generalized parking functions. 
For example, $(n,m)$-parking functions allow the $n$ cars to park in a line of $n\leq m$ parking spots and are counted by $(n-m+1)(n+1)^{m-1}$ \cite{nmParkingFunctions}.
Another generalization of parking functions given in \cite{xParkingFunctions}, known as  $\textbf{x}$-parking functions, are defined by generalizing the necessary and sufficient condition so that given $\alpha\in PP_n$ and a vector $\textbf{x}=(x_1,\dots,x_n)\in \Z^n$, $\alpha$ is an $\textbf{x}$-parking function if its increasing rearrangement $\beta=(b_1,\dots,b_n)$ satisfies $b_i\leq x_1+\cdots+ x_i$ for each $i$. For a survey of classical parking functions and their generalizations, we refer the reader to \cite{YanSurvey}.

In this paper, we study a new generalization of parking functions, introduced by Baumgardner in \cite{BaumgardnerHonorsContract}, called {\bf{Naples parking functions}}. In this generalization, the parking rule for the parking preference $\alpha=(a_1,a_2,\ldots,a_n)$ is as follows. Car $c_i$ drives to its preferred parking spot $a_i$, and if the spot is empty $c_i$ parks there. Otherwise, $c_i$ first checks back to see if parking spot $a_i-1$ (the one directly behind its preferred parking spot) is available. 
If spot $a_i-1$ is empty and $a_i-1\geq 1$, $c_i$ parks there. Otherwise, $c_i$ continues east and parks in the first unoccupied spot. If under this new parking rule the parking preference $\alpha$ allows all cars to park, then we call $\alpha$ a Naples parking function. 
We extend this parking rule by allowing a car that finds its preferred parking spot occupied to look back up to $k$ spaces, for $0\leq k< n$. The car backs up one space at a time and parks in the first spot available. If none of the $k$ spaces before its preferred parking spot are available, then the car continues east past its preferred spot and parks in the first available spot. If under the parking preference $\alpha$ all cars can park using this new parking rule, then we say that $\alpha$ is a $k$-Naples parking function of length $n$ and we denote this set by $PF_{n,k}$. Then
$PF_{n,0}=PF_n$, $PF_{n,1}$ is the set of Naples parking functions, and $PF_{n,k-1}\subseteq PF_{n,k}$ for all $0\leq k< n$.

Our first main result provides a recursive formula for the number of $k$-Naples parking functions of length $n$.

\begin{theorem}\label{thm:mainrecurssion}
If $k,n\in \N$ with $0\leq k\leq n-1$, then the number of $k$-Naples parking functions of length $n+1$ is counted recursively by
$$|PF_{n+1,k}| =  \sum_{i=0}^{n} \binom{n}{i} \min((i+ 1) +k, n +1)|PF_{i,k}|(n-i+1)^{n-i-1}.$$
\end{theorem}

Given a recurrence, there are well-established ways in which one can develop closed formulas. However, these techniques cannot be applied to the recursive formula in Theorem \ref{thm:mainrecurssion} since simplifying the recursion by removing factors yields recurrences that enumerate combinatorial objects for which there are no known formulas. 
For example, if we simplify the recursion to $a_{n+1} = \sum_{i = 0}^n {n\choose i}a_i$ with seed values $a_0 = a_1 = 1$, it yields the Bell numbers\footnote{\href{http://oeis.org/A000110}{OEIS \textcolor{blue}{A000110}}.}, for which there is no known closed formula. If we incorporate the factor $(n-i+1)^{n-i-1}$ to the simplified recurrence, then $a_{n+1} = \sum_{i = 0}^n {n\choose i}(n-i+1)^{n-i-1}a_i $ counts the number of forests of trees on $n$ labeled nodes\footnote{\href{http://oeis.org/A001858}{OEIS \textcolor{blue}{A001858}}.}, for which there is also no known closed formula. Lastly, incorporating the term $\min((i+ 1) +k, n +1)$ into the previous recurrence yields the recursion presented in Theorem~\ref{thm:mainrecurssion}.

In light of the fact that these subsets of the set $PF_{n,k}$ do not have closed formulas for their size, we focus our study on the growth of $|PF_{n,k}^*|:=|PF_{n,k}\setminus PF_{n,k-1}|$ as we fix $n$ and increase $k$ from $1$ to $n$, and where $PF_{n,0}^*=\emptyset$. Experimental evidence suggests that $|PF_{n,k}^*|$ is largest when $k=1$, which corresponds to the number of parking preferences gained by changing the parking rule defining classical parking functions to that defining Naples parking functions.
For $n=25,50,75,100$ and $0\leq k\leq n$, we plot the size of $PF^*_{n,k}$ in Figure~\ref{fig:expevidence}.

\begin{figure}[h]
    \centering
     \begin{subfigure}[b]{0.4\textwidth}
     \centering
    \begin{tikzpicture}
    \node[inner sep=0pt](a1) at (0,0)
    {\includegraphics[width=2.5in]{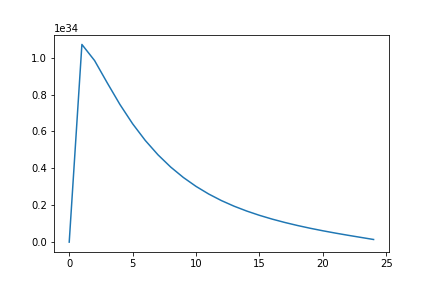}};
    \node[inner sep=0pt](a2) at (-3.5,.25) {$|PF_{n,k}^*|$};
    \node[inner sep=0pt](ak) at (.25,-2) {$k$};
    \end{tikzpicture}
     \caption{$n=25$}
     \end{subfigure}
     \qquad\qquad
     \begin{subfigure}[b]{0.4\textwidth}
     \centering
    \begin{tikzpicture}
    \node[inner sep=0pt] (pic1) at (0,0)
    {\includegraphics[width=2.5in]{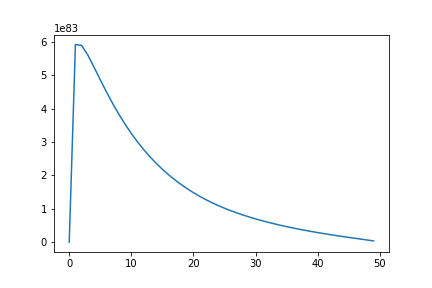}};
    \node[inner sep=0pt] (P) at (-3.5,.25) {$|PF_{n,k}^*|$};
    \node[inner sep=0pt] (k) at (.25,-2) {$k$};
    \end{tikzpicture}
     \caption{$n=50$}
     \end{subfigure}
     \\
     \begin{subfigure}[b]{0.4\textwidth}
     \centering
    \begin{tikzpicture}
    \node[inner sep=0pt] (pic1) at (0,0)
    {\includegraphics[width=2.5in]{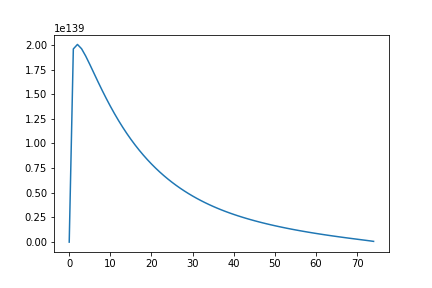}};
    \node[inner sep=0pt] (P) at (-3.5,.25) {$|PF_{n,k}^*|$};
    \node[inner sep=0pt] (k) at (.25,-2) {$k$};
    \end{tikzpicture}
     \caption{$n=75$}
     \end{subfigure}
      \qquad
      \qquad
    \begin{subfigure}[b]{0.4\textwidth}
     \centering
    \begin{tikzpicture}
    \node[inner sep=0pt] (pic1) at (0,0)
    {\includegraphics[width=2.5in]{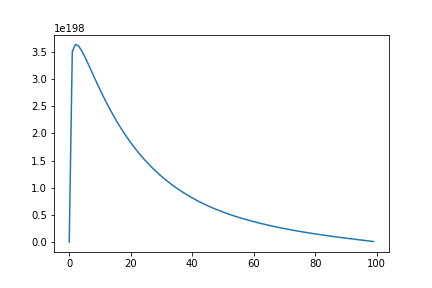}};
    \node[inner sep=0pt] (P) at (-3.5,.25) {$|PF_{n,k}^*|$};
    \node[inner sep=0pt] (k) at (.25,-2) {$k$};
    \end{tikzpicture}
     \caption{$n=100$}
     \end{subfigure}
     \caption{Plots for $|PF_{n,k}^*|$ for varying values of $n$ and with $1\leq k\leq n$. The scale of the $y$-axis is scaled by a factor of $10^{34}$, $10^{83}$, $10^{139}$, and $10^{198}$, when $n=25, 50, 75$, and $100$, respectively.}
    \label{fig:expevidence}
\end{figure}
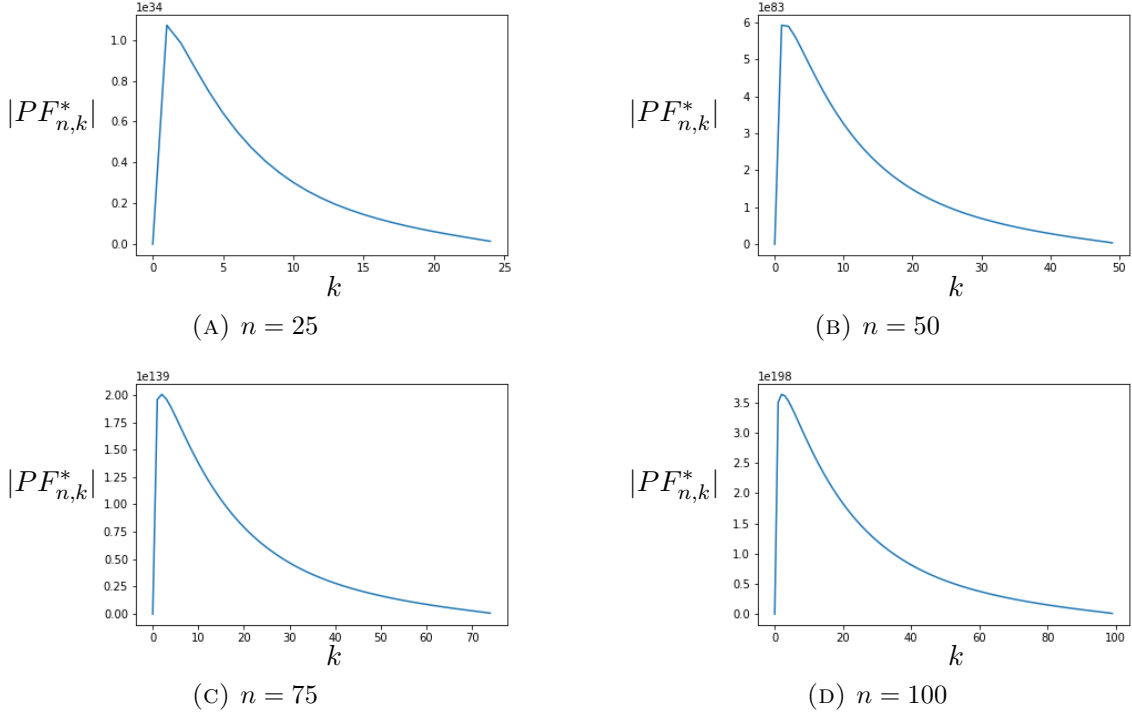

Given this observation, Naples parking functions are of particular interest. Our next main result gives a necessary and sufficient condition to characterizing Naples parking functions.

\begin{theorem}\label{TheBigBoi}
Fix $n\in\N$. Let $\alpha = (a_1, a_2, \dots, a_n)\in PP_n$, and define $\Tau:PP_n\to PP_n$ as $\Tau(\alpha)=(\tau(a_1), \tau(a_2), \dots,\tau(a_n))$, where $\tau(a_i)$ is defined

\[
\tau(a_i) =\begin{cases}  
      a_i & \text{if $i=1$, or if $a_i=1$, or if $a_i\neq 1$ and }  a_i \neq \tau(a_j)  \text{ for all }1\leq j < i\leq n\\
      a_i - 1 & \text{if $a_i\neq 1$ and $a_i = \tau(a_j)$ for some } 1\leq j < i\leq n.
   \end{cases}
\]
Then $\alpha$ is a Naples parking function if and only if $\Tau(\alpha)$ is a parking function.
\end{theorem}

It is known that every rearrangement of the entries of a parking function is also a parking function. However, this is not true for $k$-Naples parking functions that are not parking functions. Therefore,
we study decreasing $k$-Naples parking functions of length $n$, those whose entries are in weakly-decreasing order, and give a bijection from this set to a set of decreasing lattice paths of length $2n$, which we call $k$-lattice paths. These lattice paths are a particular family of signature Dyck paths and we enumerate certain families of them. We note that signature Dyck paths were defined by Cellabos and Gonz\'alez D'Le\'on, but in general there are no known closed formulas enumerating these combinatorial objects \cite{D'LeonandCeballos}.

\begin{theorem}\label{DecreasingNaplesiffDyck}
If $n,k\in \N$ with $1\leq k\leq n$, then the set of decreasing $k$-Naples parking functions of length $n$ and the set of $k$-lattice paths of length $2n$ are in bijection.
\end{theorem}

This paper is organized as follows. Section \ref{sec:Preliminaries} gives a precise definition of the $k$-Naples parking functions,  some illustrative examples, and some preliminary results.
In Section \ref{sec:RecursiveCount}, we prove \mbox{Theorem \ref{thm:mainrecurssion}}, thereby providing a formula for computing the number of $k$-Naples parking functions for any length $n$. Then, in Section \ref{sec:TauCharacterization} and \ref{sec:DecreasingLatticePaths}, we prove Theorem~\ref{TheBigBoi} and \ref{DecreasingNaplesiffDyck}, respectively. For the interested reader, we scatter open problems throughout.

\section{Background and preliminaries}\label{sec:Preliminaries}
Given an integer $1\leq k\leq n-1$, we consider a new parking rule for the parking preference $\alpha=(a_1,a_2,\ldots,a_n)$. Car $c_i$ drives to its preferred parking spot $a_i$, and if the spot is occupied, then car $c_i$ first checks back one spot at a time to see if any of the parking spots in the set $A_{i,k}:=\{a_i-1, a_i-2,\ldots, a_i-k\}\cap [n]$ are available. Note the intersection is present as cars cannot look back past the first parking spot. If any of the spots in $A_{i,k}$ are empty, then $c_i$ parks in the available spot $a_j\in A_{i,k}$ which is closest to its preferred parking spot $a_i$. 
If all of the parking spots in the set $A_{i,k}$ are occupied, then $c_i$ proceeds east until it reaches the first unoccupied parking spot after $a_i$. If under this new parking rule the parking preference $\alpha$ allows all cars to park, then we call $\alpha$ a $k$-Naples parking function of length $n$. 
We denote the set of all $k$-Naples parking functions of length $n$ by $PF_{n,k}$. 
We illustrate these definitions below.

\begin{example}
Consider the parking preference $(1,3,3,2)$. Notice that this parking preference is both a parking function and a Naples parking function. However, the order in which the cars park varies, depending on if we are using the classical parking rule or the Naples parking rule. According to the classical parking rule, we have that $c_1$ parks in the first space, $c_2$ in the third space, then $c_3$, finding the third space occupied, continues east and parks in the fourth space, and $c_4$ parks in its preferred second space. This is illustrated in Figure \ref{fig:cpf}. In contrast, according to the Naples parking function rule, we have that $c_1$ parks in the first space, $c_2$ in the third space, then $c_3$, finding the third space occupied, looks back a space and parks in the unoccupied second space. Finally, $c_4$ finds the second space occupied and continues east until it parks in the unoccupied fourth space. This is illustrated in Figure \ref{fig:npf}.
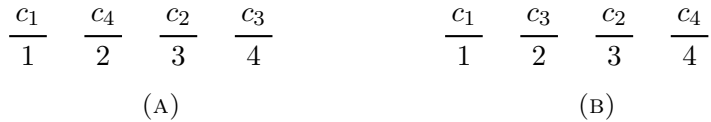
\begin{figure}[h]
\centering
\begin{subfigure}[b]{.25\textwidth}
\begin{tikzpicture}
\draw[ thick] (0,0) to node[below]{1} node[above]{$c_1$} (0.5,0);
\draw[ thick] (1,0) to node[below]{2} node[above]{$c_4$} (1.5,0);
\draw[ thick] (2,0) to node[below]{3} node[above]{$c_2$} (2.5,0);
\draw[ thick] (3,0) to node[below]{$4$} node[above]{$c_3$}(3.5,0);
\end{tikzpicture}
\caption{}
                \label{fig:cpf}
\end{subfigure}
\qquad\qquad
\begin{subfigure}[b]{.25\textwidth}
\begin{tikzpicture}
\draw[ thick] (6,0) to node[below]{1} node[above]{$c_1$} (6.5,0);
\draw[ thick] (7,0) to node[below]{2} node[above]{$c_3$} (7.5,0);
\draw[ thick] (8,0) to node[below]{3} node[above]{$c_2$} (8.5,0);
\draw[ thick] (9,0) to node[below]{$4$} node[above]{$c_4$}(9.5,0);
\end{tikzpicture}
\caption{}
                \label{fig:npf}
\end{subfigure}
\caption{Illustration of order in which cars with preference $(1,3,3,2)$ park under the classical parking rule (left) and under the Naples parking rule (right).}\label{fig:example1}
\end{figure}
\end{example}

We observe that for any parking preference of length $n$ there is a  maximum of $n-1$ steps backward that a car can take from its preferred parking space. Moreover, if each car can take up to $n-1$ steps backwards then each car is able to check each of the $n$ spaces and all the cars park. Namely, 
\begin{align}
    |PF_{n,k}|&=|PP_n|, \mbox{ whenever $k\geq n-1$}.\label{eq:pf=pp}
\end{align} 
In Table \ref{tab:PFnkchart}, we provide the cardinalities\footnote{Sequences in Table \ref{tab:PFnkchart} were computed using {\color{blue} \url{https://github.com/andresramos5/Naples-Parking-Function.git}}.} of the sets $PF_{n,k}$ for varying $k\leq n$. 

\begin{table}[h] 
    \centering
    \resizebox{\textwidth}{!}{
    \begin{tabular}{|c||c|c|c|c|c|c|c|c|}\hline
     $n$  & $k=0$ &$k = 1$ & $k = 2$ & $k = 3$ & $k = 4$ & $k = 5$ & $k=6$ & $k=7$\\
         \hline\hline
      1   & \textbf{1} &&&&&&&\\\hline
      2 &3 &\textbf{4} &&&&&&\\\hline
      3 & 16 & 24 & \textbf{27}& & & &&\\\hline
      4 & 125&203&240& \textbf{256} &&&&\\\hline
     5 & 1,296& 2,225& 2,731& 3,000& \textbf{3,125}&&&\\\hline
     6 & 16,807& 30,067& 38,034& 42,689& 45,360& \textbf{46,656}&&\\\hline
     7 & 262,144& 484,071& 627,405& 717,051& 773,081& 806,736 & \textbf{823,543}&\\\hline
     8 & 4,782,969 & 9,057,316 & 11,976,466 & 13,902,752 & 15,170,350 & 16,000,823 & 16,515,072 & \textbf{16,777,216}\\\hline
    \end{tabular}}
    \caption{The cardinality of $PF_{n,k}$. Numbers in \textbf{bold} are $n^n$, which count the cardinality of $PF_{n,k}$ for $k\geq n-1$. The first column, where $k=0$ is \mbox{$|PF_n|=(n+1)^{n-1}$ }.}
    \label{tab:PFnkchart}
\end{table}

From the sequences in Table \ref{tab:PFnkchart}, the On-line Encyclopedia of Integer Sequences (OEIS) only catalogs the sequences $(PF_{n,0})_{n\in \N}$ and $(PF_{n,n-1})_{n\in \N}$, which are the number of parking functions and the number of parking preferences, respectively. Thus, it appears that many of the sequences associated with $k$-Naples parking functions have not been studied. However, notice that the difference of the diagonal and subdiagonal in the table arising from the computation of \mbox{$|PF_{n,n-1}|-|PF_{n,n-2}|$} yields the sequence $1,3,16,125,\ldots$, which is precisely the number of parking functions. In fact there is a bijection between $PF^*_{n,n-1}=PF_{n,k}\setminus PF_{n,k-1}$ and $PF_{n-1}$, which we discuss in Theorem~\ref{thm:n-2NaplesCardinality}. As a consequence of this result, we establish a closed formula for the number of  $(n-2)$-Naples parking functions of length $n$, as presented in Corollary \ref{cor:closedformula1}.

First, in order to formally identify the bijection between $PF^*_{n,n-1}$ and $PF_{n-1}$, we begin with the following observation about the set $PF^*_{n,n-1}$ of parking preferences that are not $k$-Naples parking functions for $k<n-1$.

\begin{lemma}\label{lem:Usedinthetheorem}
If $\beta = (b_1, b_2, \dots, b_n) \in PF_{n,n-1}^*$, then $b_n = n$.
\end{lemma}
\begin{proof}
 By way of contradiction, assume $\beta = (b_1, b_2, \dots, b_n) \in PF_{n,n-1}^*$ and $b_n \neq n$. All of the cars can park because $\beta$ is a $(n-1)$-Naples parking function. This implies that when you get to the last car, $c_n$, only one spot is open. By assumption, $c_n$'s preference satisfies $1\leq b_n \leq n-1$. If $c_n$ arrives to its preferred space and finds it occupied, it first checks backwards. The maximum number of steps back that $c_n$ can take is $(n-1) -1 = n-2 < n-1$. If $c_n$ takes $n-2$ steps back it has checked all the spaces behind its preferred space. Therefore, if the remaining empty space is behind $c_n$'s preferred space then $c_n$ finds it and parks there. If not, $c_n$ can move forward and check all the remaining spaces to find the empty one. Thus, $c_n$ can park with only $n-2$ steps back and $\beta \in  PF_{n,n-2}$. This contradicts our assumption that $\beta \in PF_{n,n-1}^*=PF_{n,n-1}\setminus PF_{n,n-2}$. Thus, $b_n = n$.
\end{proof}

Lemma \ref{lem:Usedinthetheorem} aids in establishing the following result.

\begin{theorem}\label{thm:n-2NaplesCardinality}
Let $\al= (a_1, a_2, \dots, a_{n-1}) \in PF_{n-1} $, and define $\Psi: PF_{n-1} \to PF_{n,n-1}^*$ by
\[
\Psi(\al) = (\psi(a_1), \dots, \psi(a_{n-1}), n),
\]
where $\psi(a_i) =n + 1 - a_i$. Then $\Psi$ is a bijection between $PF_{n-1}$ and $PF_{n,n-1}^*$. 
\end{theorem}
\begin{proof}
Since $\alpha=(a_1, a_2, \dots, a_{n-1})$ is a parking function of length $n-1$, we have that $a_i\in [n-1] $ for all $i$ so that $\psi(a_i)\in [n]$. Thus, $\Psi(\alpha)\in PF_{n,n-1}=PP_n$. To verify that $\Psi(\alpha)\notin PF_{n,n-2}$, consider the setup outlined below.

Denote the $n-1$ cars with parking preferences given by $\alpha$ as $c_1,\dots, c_{n-1}$, and denote the $n$ cars with parking preferences given by $\Psi(\alpha)$ as $d_1,\dots,d_n$ in order to distinguish between the two. Now, consider $\Psi$ in the following way. For parking function $\alpha\in PF_{n-1}$ arrange for the car $c_i$ to park on a one-way street labeled 1 to $n$ from east to west, where they start driving from the eastern-most space labeled $1$ to their desired space $a_i$ and then proceed west if their desired space is occupied. See the red labeling of spaces in Figure \ref{fig:redblackspaces}. Thus, $c_i$ has parking preference $\psi(a_i)=n+1-a_i$ on a lot labeled 1 to $n$ from west to east.  See the black labeling of spaces in Figure \ref{fig:redblackspaces}.

\begin{figure}[h]
    \centering
    \begin{tikzpicture}[scale=1.5]
\draw[ thick] (0,0) to node[below]{1} (0.5,0);
\draw[ thick] (1,0) to node[below]{2} node[below=.5cm, red]{$n-1$} (1.5,0);
\draw[ thick] (2,0) to node[below]{3} node[below=.5cm, red]{$n-2$}(2.5,0);
\node at (3.25,0) {$\dots$};
\draw[ thick] (4,0) to node [below]{$n-2$} node[below=.5cm, red]{$3$}  (4.5,0);
\draw[ thick] (5,0) to node [below]{$n-1$} node[below=.5cm, red]{$2$}(5.5,0);
\draw[ thick] (6,0) to node [below]{$n$} node[below=.5cm, red]{$1$} (6.5,0);
\end{tikzpicture}
    \caption{Labeling the parking spaces in two distinct ways.}
    \label{fig:redblackspaces}
\end{figure}
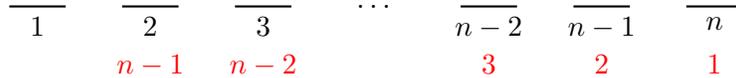

Since $\alpha\in PF_{n-1}$, the $n-1$ cars park in the red labeled spaces $1$ to $n-1$, moving from east to west, or the in black labeled spaces 2 to $n$, moving west to east. Moreover, car $d_i$ with parking preference $\psi(a_i)=n + 1 - a_i$ parks in precisely the same spot as $c_i$, whenever $1\leq i\leq n-1$, since $d_i$ proceeds to black space $n + 1 - a_i$, which is just red space $a_i$, and then proceeds west to the first available spot if it is unoccupied. Since $\alpha\in PF_{n-1}$, for any $1\leq i\leq n-1$ we know that both $c_i$ and $d_i$ never need to check a space further west than the black spot at position 2, i.e. the red spot at position $n-1$, so each car $d_i$ checks at most $n-2$ spaces behind its preferred spot. Thus, cars $d_1,\dots, d_{n-1}$ park in (black) spaces $2,\dots, n$. The car $d_n$ must have preference $n$ by Lemma~\ref{lem:Usedinthetheorem}, implying that the last preference of $\Psi(\alpha)$ is always $n$ 
and $\Psi(\alpha)\notin PF_{n,n-2}$. 

Next, observe that $\psi$ is an involution since 
\[
(\psi\circ\psi)(a_i)=n + 1 -(n + 1 - a_i)=a_i.
\]
Thus, $\Psi$ is invertible, which implies it is a bijection.
\end{proof}
Now we provide a closed formula for the number of $(n-2)$-Naples parking functions of length $n$. 
\begin{corollary}\label{cor:closedformula1}
If $n\geq 2$, then $|PF_{n, n-2}| = n^n - n^{n-2}$.
\end{corollary}
\begin{proof}
By Theorem \ref{thm:n-2NaplesCardinality}, the set of $(n-1)$-Naples parking functions that are not $(n-2)$-Naples has cardinality $n^{n-2}= ((n-1) +1)^{(n-1)-1} =|PF_{n-1}|$. Moreover, since $PF_{n,n-2}$ and $PF_{n,n-1}^*$ are disjoint, we have that $$|PF_{n,n-2}|+|PF_{n,n-1}^*|=|PF_{n,n-1}|=n^n.$$
Therefore, $|PF_{n,n-2}|=n^n-n^{n-2}$ as desired.
\end{proof}

Having found closed formulas for $|PF_{n,n-1}|$ and $|PF_{n,n-2}|$,  
in the next section we present a recursive formula to count the number of $k$-Naples parking functions for all $1\leq k\leq n-3$.

\section{Counting Naples Parking Functions Recursively}\label{sec:RecursiveCount}
In this section, we begin by introducing a recursive formula for the number of parking functions, first appearing in the work of Konheim and Weiss \cite[Equation (2.4), Lemma 1]{KonheimAndWeiss}. For ease of reference, we provide an independent proof of this result and then generalize this recursion so that it counts $k$-Naples parking functions.

\begin{theorem} \label{Recursive PF} The number of parking functions of size $n+1$ is recursively counted by the formula
\begin{align*}
|PF_{n+1}|&=\sum_{i = 0}^n \binom{n}{i}(i+1)^{i}(n-i+1)^{n-i-1}.
\end{align*}
\end{theorem}

\begin{proof}
We proceed by counting the number of parking functions of length $n+1$ given that car $n+1$ can park in the spot $i+1$ for $i=0,1,2,\ldots,n$. Let $S\subseteq \{c_1,c_2,\ldots,c_n\}$ consist of the  cars parked to the left of the $i+1$ parking space, while the cars that park to the right of the $i+1$ spot consist of the complement of $S$. Observe that there are $\binom{n}{i}$ ways to select the subset $S$. The number of ways of assigning parking preferences to the cars in $S$ so that they park before spot $i+1$ is precisely $|PF_i|$. Now, we count the number of ways of assigning parking preferences to the $n-i$ cars found to the right of spot $i+1$ so that they park in the parking spots $i+2$ to $n+1$. This is given by $|PF_{(n+1) - (i + 1)}| = |PF_{n-i}|$ since the cars do not park in any of the first $i+1$ spots. Finally, there are $i+1$ possible parking preferences that allow $c_{n+1}$ to park in spot $i+1$. Thus, the number of parking functions of length $n+1$ where car $c_{n+1}$ parks in spot $i+1$ is given by $\binom{n}{i}|PF_i||PF_{n-i}|(i+1)$. Accounting for all possible values of $i$ yields 
\begin{align*}
|PF_{n+1}|& =  \sum_{i = 0}^n \binom{n}{i}|PF_{i}||PF_{n-i}|(i+1)= \sum_{i = 0}^n \binom{n}{i}(i+1)^{i}(n-i+1)^{n-i-1},
\end{align*}
as desired.
\end{proof}

Observe that in order to generalize the recursive formula in Theorem~\ref{Recursive PF} to count $k$-Naples parking functions, we need to modify it by taking into account the new rule that allows cars to look back up to $k$ spots in search for an available one. In this case, if we want car $c_{n+1}$ to park in spot $i+1$, then we must only count the number of parking preferences that allow $n-i$ cars to park in parking spots $i+2$ to $n+1$ without backing up to park in spot $i+1$. Equivalently, we consider introducing an empty parking spot, numbered $0$, to the left of $1$ and counting the number of $k$-Naples that would leave that spot open. We refer to this subset of $k$-Naples parking functions as contained parking functions.

\begin{definition}
The set of \textbf{contained parking functions} $B_{n,k}$ is the set of all $k$-Naples parking functions of length $n$ such that if cars $c_1, \ldots, c_{i-1}$ have already filled spaces $1,\dots,a_i$, then there is no car $c_i$ with a parking preference $1\leq a_i\leq k$.
\end{definition}
We call this set the contained parking functions because if you were to introduce more available spots to the ends of the parking lot (before the first spot and/or after the $n$th spot), the $n$ cars only park in spots $1,\dots,n$, assuming their parking preferences were between $1,\dots,n$. 

\begin{example}
We let the parking lot be represented by a number line of integers and consider the $2$-Naples parking function $\alpha = (4,4,2,3)$, whose cars park as depicted in Figure~\ref{fig:onnumberline}.
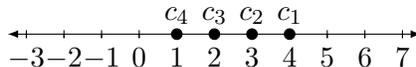
\begin{figure}[h]
    \centering
    \begin{tikzpicture}[scale=0.5]
\path [draw=black, fill=black] (1,0) circle (4pt);
\path [draw=black, fill=black] (2,0) circle (4pt);
\path [draw=black, fill=black] (3,0) circle (4pt);
\path [draw=black, fill=black] (4,0) circle (4pt);
\draw[latex-latex] (-3.5,0) -- (7.5,0) ;
\foreach \x in  {-3,-2,-1,0,1,2,3,4, 5, 6, 7}
\draw[shift={(\x,0)},color=black] (0pt,3pt) -- (0pt,-3pt);
\foreach \x in {-3,-2,-1,0,1,2,3,4, 5, 6, 7}
\draw[shift={(\x,0)},color=black] (0pt,0pt) -- (0pt,-3pt) node[below]
{$\x$};
\node[coordinate,label=above:{$c_4$}] at (1, 0) {};
\node[coordinate,label=above:{$c_3$}] at (2, 0) {};
\node[coordinate,label=above:{$c_2$}] at (3, 0) {};
\node[coordinate,label=above:{$c_1$}] at (4, 0) {};
\end{tikzpicture}
    \caption{Parking position of cars with parking preference $\alpha=(4,4,2,3)$.}
    \label{fig:onnumberline}
\end{figure}

If $\beta = (4,2,2,2)$, then the cars park as illustrated in Figure \ref{fig:numberline2}. \begin{figure}[h]
    \centering
    \begin{tikzpicture}[scale=0.5]
\path [draw=black, fill=black] (1,0) circle (4pt);
\path [draw=black, fill=black] (2,0) circle (4pt);
\path [draw=black, fill=black] (4,0) circle (4pt);
\path [draw=black, fill=black] (0,0) circle (4pt);
\draw[latex-latex] (-3.5,0) -- (7.5,0) ;
\foreach \x in  {-3,-2,-1,0,1,2,3,4, 5, 6, 7}
\draw[shift={(\x,0)},color=black] (0pt,3pt) -- (0pt,-3pt);
\foreach \x in {-3,-2,-1,0,1,2,3,4, 5, 6, 7}
\draw[shift={(\x,0)},color=black] (0pt,0pt) -- (0pt,-3pt) node[below]
{$\x$};
\node[coordinate,label=above:{$c_4$}] at (0, 0) {};
\node[coordinate,label=above:{$c_3$}] at (1, 0) {};
\node[coordinate,label=above:{$c_2$}] at (2, 0) {};
\node[coordinate,label=above:{$c_1$}] at (4, 0) {};
\end{tikzpicture}
    \caption{Parking position of cars with parking preference $\alpha=(4,2,2,2)$}
    \label{fig:numberline2}
\end{figure}
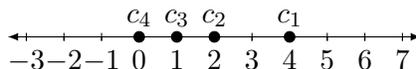

Hence $\beta = (4,2,2,2)\notin B_{4,2}$, because $c_4$ was able to look back past spot $1$ and park in spot $0$, leaving spot $3$ empty. Thus, the cars' final parking positions are not contained in spots 1 through 4. However, $\beta\in PF_{4,2}$, because under normal conditions it would not check any spot west of $1$ and car $c_4$ would park in spot $3$.
\end{example}

With these definitions in hand, we now determine the number of contained parking functions. Our proof adapts Pollak's technique to establish that $|PF_n|=(n+1)^{n-1}$ \cite{PollakCircleArgument}. 

\begin{lemma} \label{lem:BnkPFnequinumerous}
If $n\in N$ and $k \in \{0,1,\ldots, n\}$, then $|B_{n,k}|= (n+1)^{n-1}$.
\end{lemma}

\begin{proof}
Consider $\beta\in PP_n$. Each car can check up to $k$ spaces behind their preferred parking spot if it is occupied and only proceeds forward if all the spots they are allowed to check behind them are occupied. Let us arrange these parking spaces clockwise on a circle instead of on a line and introduce a space 0 between 1 and $n$. Now, if a car's preferred parking space is occupied, it checks up to $k$ spaces counterclockwise from its preferred parking space and proceeds clockwise if those spots are also occupied. Based on this set up, any parking preference of length $n$ allows all cars to park and leaves one space unoccupied. Observe that the parking preference is an element of $B_{n,k}$ if and only if the cars park in a way that leaves spot $0$ unoccupied.

To count the number of ways of assigning $n$ cars parking preferences on the circle, first count the number of ways to assign $n+1$ preferences to $n$ cars, which is $(n+1)^n$. For each parking preference, $\beta=(b_1,\dots,b_n)$ exactly one ``clockwise rotation'' of the wheel by an integer $j$, i.e. the parking preference $(b_1+j,\dots,b_n+j)\text{ }(\text{mod} (n+1))$, leaves the spot $n+1$ unoccupied. Thus, there are \mbox{$ \frac{(n+1)^n}{n+1}=(n+1)^{n-1}$} elements in $B_{n,k}$.
\end{proof}

Note that Lemma \ref{lem:BnkPFnequinumerous} implies that the sets $B_{n,k}$ and $PF_n$ are equinumerous. 
For clarity's sake, it is important to note that the argument used in the proof of Theorem \ref{lem:BnkPFnequinumerous} cannot be used to count $k$-Naples parking functions, because in addition to the contained $k$-Naples parking functions counted in this argument, there are parking functions with $k$ steps back that occupy spot $0$ on the circle. 
For example, $(1,1,1)\in PF_{3,1}$, but since cars first check spots behind their preferred parking spot, space 0 on the circle is occupied by the second car. Therefore, there are $k$-Naples parking functions that are not counted using this argument. Moreover, for small values of $n$ we found that not only are the sets \containedpf and \standardpf equinumerous,  but they also share specific characteristics. 
To describe these characteristics, we consider  $\alpha=(a_1,a_2,\ldots,a_n)\in PP_n$, and for all $1\leq i\leq n-1$, we~say 
\begin{itemize}
    \item $i$ is an ascent if $a_i<a_{i+1}$, 
    \item $i$ is a descent if $a_i>a_{i+1}$,  and 
    \item $i$ is a tie if $a_i=a_{i+1}$.
\end{itemize} 
Experimentally, the number of ascents, descents and ties in the set \containedpf are the same as the number of  ascents, descents and ties of $PF_n$, respectively. The enumeration of  descents and ties of parking functions was studied in \cite{pfdt}. These observations lead us naturally to the following open problem.

\begin{problem}
Find a bijection between $B_{n,k}$ and $PF_n$ that preserves the number of ascents, descents and ties. 
\end{problem}

We now use the set of contained parking functions to give a recursive formula for the number of $k$-Naples parking functions of length $n$.

\begin{reptheorem}{thm:mainrecurssion}
If $k,n\in \N$ with $0\leq k\leq n-1$, then the number of $k$-Naples parking functions of length $n+1$ is counted recursively by
\begin{align}
    |PF_{n+1,k}| =  \sum_{i=0}^{n} \binom{n}{i} \min((i+ 1) +k, n +1)|PF_{i,k}|(n-i+1)^{n-i-1}.\label{eq:mainrecform}
\end{align}
\end{reptheorem}
\begin{proof}
As in Theorem \ref{Recursive PF}, we now construct a recursion that counts the number of ways that $n+1$ cars can park given that car $c_{n+1}$ parks in the spot $i+1$ for $0 \leq i \leq n$. 
Let $S\subseteq \{c_1,c_2,\ldots,c_n\}$ consist of the  cars parked to the left of the $i+1$ parking space, while the cars that park to the right of the $i+1$ spot consist of the complement of $S$.
There are $\binom{n}{i}$ ways to choose the subset $S$. Given $S$, the number of ways of assigning parking preferences to the cars in $S$ which allow them to park in the first $i$ spaces is the number of $k$-Naples parking functions of length $i$, $|PF_{i,k}|$. 
Recall that spot $i+1$ must remain empty so that car $c_{n+1}$ can park there. Since cars can check up to $k$ spots behind their preferred parking spot, we must be careful to only count the parking preferences for cars in $\{c_1,\dots,c_n\}\setminus S$ which ensure that they do not park in spot $i+1$.
Fortunately, the set of parking preferences we just described is exactly $B_{n-i,k}$, and by Lemma \ref{lem:BnkPFnequinumerous}, we know that $|B_{n-i,k}| = (n-i+1)^{n-i-1}$. 
Lastly, we count how many possible parking preferences allow car $c_{n+1}$ to park in spot $i+1$. Since car $c_{n+1}$ can check up to $k$ spots behind its preferred spot, $a_{n+1}$, car $c_{n+1}$ parks in spot $i+1$ only if  
$1\leq a_{n+1}\leq i+1+k$. Also $i+1+k\leq n+1$, as there are only $n+1$ parking spots.
Thus, the number of ways of assigning a parking preference to $c_{n+1}$ so that it parks in spot $i+1$ is $\text{min}((i + 1) + k,n+1)$. The result follows from accounting for all possible values of $i$, which yields
\[|PF_{n+1,k}| =  \sum_{i=0}^{n} \binom{n}{i} \min((i+ 1) +k, n +1)|PF_{i,k}|(n-i+1)^{n-i-1}.\qedhere\]
\end{proof}

Evaluating Equation \eqref{eq:mainrecform} at $k=0$ recovers the recurrence from Theorem \ref{Recursive PF}.\\

In this section, we obtained a closed formula for the number of $k$-Naples parking functions length $n$ only in the special cases where $k=n-1$ or $n-2$ and provided a recursive formula for all other values of $0\leq k\leq n-3$. 
It remains an open problem to determine closed formulas $|PF_{n,k}|$, but as we discussed in the introduction, such a formula is beyond the scope of our current study. However, we note that \[|PF_{n,k}| = |PF_n| + |PF_{n,k}\setminus PF_n| = (n+1)^{n-1} + X,\] and, by Lemma \ref{lem:BnkPFnequinumerous}, we know that $|PF_n| = |B_{n,k}|$. Therefore, we can write $|PF_{n,k}| = |B_{n,k}| + |B_{n,k}^c|$ where $B_{n,k}^c$ is the complement of $B_{n,k}$ in $PP_n$. Thus, $|PF_{n,k}\setminus PF_n| = |B_{n,k}^c|$. Thus, finding a closed formula for $|B_{n,k}^c|$ is just as difficult as finding a closed formula for $|PF_{n,k}|$. This motivates our next open problem.

\begin{problem}
Find a closed formula to count the number of elements in $B_{n,k}^c$.
\end{problem}

\section{Characterization of Naples Parking Functions}\label{sec:TauCharacterization}

In this section, we specialize the parameter $k=1$ and focus our study on the set $PF_{n,1}$, i.e. the set of Naples parking functions of length $n$. 
The question of interest is: \textit{Given a parking preference, how can we determine if it is a Naples parking function? }
To determine whether a parking preference is a Naples parking function, we define a function which reduces the problem to checking if the image of a Naples parking function is a parking function.

\begin{definition}\label{TauDefinition}
Fix $n\in\N$ and let $\alpha = (a_1, a_2, \dots, a_n)\in PP_n$. 
We define $\Tau:PP_n\to PP_n$ as $\Tau(\alpha)=(\tau(a_1), \tau(a_2), \dots,\tau(a_n))$, where
\[
\tau(a_i) =\begin{cases} 
    a_i & \text{if  $i=1$, or if $a_i=1$, or if $a_i\neq 1$ and $a_i\neq \tau(a_j)$ for all $1\leq j<i\leq n$}\\
    a_i - 1 & \text{if $a_i\neq 1$ and $a_i = \tau(a_j)$ for some } 1\leq j < i\leq n.
   \end{cases}
\]
\end{definition}

We illustrate Definition \ref{TauDefinition} below.

\begin{example}
 Let $\al = (2,4,4,1) \in PF_{4,1}$. 
Note $\tau(a_1) = a_1 = 4$, as $i=1$. Notice that \mbox{$a_2 = 4\neq \tau(a_1)$}, so $\tau(a_2) = a_2 = 4$. Since $a_3 = 4\neq 1$ and $\tau(a_2)=a_3$ $\tau(a_3) = a_3-1=3$. Lastly,  $a_4=1$, so $\tau(a_4) = 1$. This establishes that $\Tau(\al) = (2,4,3,1)$. Note that $T(\alpha)\in PF_4$.
\end{example}

We are now ready to prove Theorem \ref{TheBigBoi}, which we restate below for ease of reference.

\begin{reptheorem}{TheBigBoi}
Fix $n\in\N$ and let $\al$ be a parking preference. Then $\alpha$ is a Naples parking function if and only if $\Tau(\alpha)$ is a parking function.
\end{reptheorem}
\begin{proof}
We first show that if $\alpha \in PF_{n,1}$, then $ \Tau(\alpha) \in PF_n$. Suppose $\alpha = (a_1, a_2, \dots, a_n) \in PF_{n,1}$ and $\Tau(\alpha)= (b_1, b_2, \dots, b_n) $. By Definition \ref{TauDefinition}, we know that for each $1\leq i\leq n$, $b_i = a_i$ or $b_i = a_i - 1$. In particular, if some car $c_i$ has preferred spot $a_i$ and that spot is taken by some car $c_j$, with $1\leq j < i$, then $b_i = a_i -1$. Otherwise, we have $b_i=a_i$. 
Since $c_i$ can park using the Naples parking rule, then there exists a spot $q$ with $a_i -1\leq b_i \leq q\leq n$ that is unoccupied. In other words, there must be an empty spot somewhere between spots $a_{i}-1$ and $n$ in order for $c_i$ to park. Because $ a_{i}-1\leq b_i$, the new preference $b_i$ ensures that the car finds an empty spot to park in, which is either at position $b_i$ or somewhere ahead of it. Thus, $c_i$ is able to park using the original parking rule. Since $i$ is arbitrary, each car $c_i$ with preference $b_i$ can park for $1\leq i \leq n$ using the original parking rule.

To show that $\Tau(\alpha) \in PF_n$ implies  $\alpha \in PF_{n,1} $ we prove the contrapositive. That is, if $\alpha \notin PF_{n,1}$ then $\Tau(\alpha) \notin PF_n$. Let $ \Tau(\alpha)=(b_1, b_2, \dots, b_n)$, where $\alpha = (a_1, a_2, \dots, a_n) \notin PF_{n,1}$. As above, $b_i = a_i$ or $b_i = a_i - 1$. Since $\alpha\notin PF_{n,1}$ there exists a car $c_i$ that cannot park using the Naples parking rule. That means that there does not exists an available spot $q$ satisfying $a_i-1\leq b_i\leq q\leq n$. Moreover, since none of these spots are available for parking using the Naples parking rule, they are also not available when parking using the original parking rule. Thus, $\Tau(\alpha)\notin PF_n$. 
\end{proof}

With the complete characterization of Naples parking functions complete, we now study their connection to Dyck paths.

\section{Connections to Decreasing Lattice Paths}\label{sec:DecreasingLatticePaths}

In this section, we introduce $k$-Lattice paths, a generalization of Dyck paths, and give a bijection between these objects and decreasing $k$-Naples parking functions. 
This result exploits the classical result which gives a correspondence between Dyck paths and decreasing\footnote{The original proof considers increasing parking functions, but the bijection holds under a slight change of indices for the decreasing parking functions.} parking functions. We end the section by connecting our main result, Theorem \ref{DecreasingNaplesiffDyck}, to signature Catalan objects, as presented in the work of Cellabos and Gonz\'alez D'Le\'on  \cite{D'LeonandCeballos}. 

In what follows, we consider decreasing rearrangements of $k$-Naples parking functions, as increasing rearrangements of $k$-Naples parking functions are not necessarily $k$-Naples. For example, $(4,1,4,3),(4,4,3,1)\in PF_{4,1}$, but $(1,3,4,4) \not\in PF_{4,1}$. 
Therefore, it is more natural to consider Dyck paths drawn from $(0,n)$ to $(n,0)$ using east and south steps. We present our formal definition below.

\begin{definition}\label{DyckPath}
For a given $n\in\N$, a \textbf{Dyck path} of length $2n$ is a lattice path from $(0,n)$ to $(n,0)$ consisting of $n$ steps by $(1,0)$ east and $n$ steps by $(0,-1)$ south such that the path never goes above the line $y=n-x$. For any south step, the number of south steps proceeding it is larger than the number of east steps preceding it. We denote the set of all Dyck paths of length $2n$ as $LP_{n}$.
\end{definition}

We now describe the bijection between decreasing parking functions and Dyck paths. Recall that a decreasing parking function is one whose entries are written in weakly-decreasing order.
Specifically, if $\alpha=(a_1,\ldots,a_n)\in PF_n$ is a decreasing parking function, then the corresponding lattice path has an east step $(i-1,a_i-1)$ to $(i,a_i-1)$ at height $a_i-1$ for each $1\leq i\leq n$, and south steps connecting these east steps so that the result is a connected path from $(0,n)$ to $(n,0)$. The fact that $\alpha$ is a decreasing parking function implies that $a_i\leq n-i+1$, hence the corresponding lattice path does not cross the line $y=n-x$.

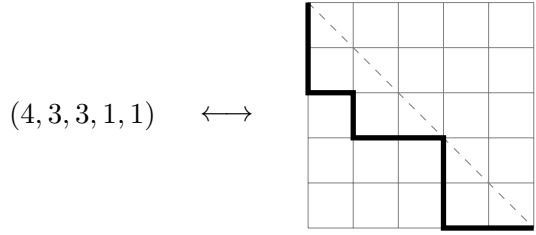
\begin{figure}[h]
    \centering
    \begin{tikzpicture}[scale=0.6]
    \node at (-5,2.5) {$(4,3,3,1,1)$};
    \node at (-1.8,2.5) {$\longleftrightarrow$};
\draw [help lines] (0,0) grid (5,5);
\draw[gray,dashed] (0,5) to (5,0);
\draw[line width=2pt] (0,5) to (0,3) to (1,3) to (1,2) to (3,2) to (3,0) to (5,0);
\end{tikzpicture}
    \caption{Dyck path corresponding to $\alpha=(4,3,3,1,1)$.}
    \label{fig:exDtoPF}
\end{figure}

\begin{example} 
Let $\alpha=(4,3,3,1,1)$ and note that its associated Dyck path has one east step at height 3, two east steps at height 2, and two east steps at height 0. Figure \ref{fig:exDtoPF} illustrates the corresponding Dyck path for $\alpha$.
\end{example}

We now consider a generalization of Dyck paths, which we call $k$-lattice paths. 

\begin{definition}
If  $n,k\in \N$ with $0\leq k\leq n-1$, then a {\bf $k$-lattice path} of length $2n$ is a lattice path from $(0,n)$ to $(n,0)$ consisting of $n$ steps east by $(1,0)$ and $n$ steps south by $(0,-1)$ such that the path does not cross the line $y=n-x+k$. We denote the set of all $k$-lattice paths of length $2n$ as $LP_{n,k}$.
\end{definition}
Notice that, $LP_{n,0}=LP_n$, which is the set of Dyck paths of length $2n$. Thus, $k$-lattice paths are a generalization of Dyck paths. Next, we present our main result, which establishes a bijection between decreasing $k$-Naples parking functions and $k$-lattice paths. Since it is well-known that there is a bijection between $LP_n$ and decreasing parking functions of length $n$, in what follows, we only consider the case where $1\leq k\leq n-1$.

\begin{reptheorem}{DecreasingNaplesiffDyck}
If $n,k\in \N$ with $1\leq k\leq n-1$, then the set of decreasing $k$-Naples parking functions of length $n$ and the set of $k$-lattice paths of length $2n$ are in bijection.
\end{reptheorem}
\begin{proof}
To establish this result, it suffices to show that given a decreasing $k$-Naples parking function we can construct a $k$-lattice path, and given a $k$-lattice path there is a corresponding decreasing $k$-Naples parking function.

Suppose that $\alpha=(a_1,a_2,\dots, a_n)$ is a decreasing $k$-Naples parking function. As in the parking function case, from $\alpha$ we construct the $k$-lattice path of length $2n$ with east steps $(i-1,a_i-1)$ to $(i,a_i-1)$, which we denote $LP(\alpha)$. 

We need only show that $a_i\leq \min(n,n+k+1-i)$ holds for all $1\leq i\leq n$, as this implies that $LP(\alpha)$ is a $k$-lattice path. Suppose that there is some $i$ such that $a_i>\min(n,n+k+1-i)$ to obtain a contradiction. For $i=1,2, \dots, k+1$, note that $\min(n,n+(k+1)-i)=n$, and because $\alpha\in PF_{n,k}$ it cannot be that  $a_i>n$. 

On the other hand, if $k+1< i\leq n$, then 
$\min(n,n+k+1-i)=n+k+1-i$, and lets assume that $a_i>n+k+1-i$. Since $\alpha$ is in decreasing order we know that $a_j\geq a_i>n+k+1-i$ for all $1\leq j<i$. 
The most optimal parking preference is $a_j=n+k+2-i$ for all $1\leq j<i$, as this maximizes the number of parking positions cars $c_1,\ldots,c_j$ can occupy. That is, it leaves the most open spots to the right of position $n+k+2-i$. In this case, cars $c_1,c_2,\ldots,c_{k+1}$ park in positions $n+k+2-i,n+k+1-i,\ldots,n+2-i$, respectively. Then, cars $c_{k+2}, c_{k+3},\ldots, c_{i-1}$ first go to their preferred parking spot, namely $n+k+2-i$, finding it occupied they back up and all of the $k$ prior spots are also full. Thus, these cars go forward and occupy the last $i-k-2$ spots numbered $n+k+3-i$ to $n$. Then car $c_i$, arriving to its preferred position, again $n+k+2-i$, finding it occupied backs up and also finds all $k$ spots behind full. It then moves forward and finds all remaining spots taken. Thus, $c_i$ is unable to park contradicting the assumption that $\alpha\in PF_{n,k}$.

We now go from an arbitrary $k$-lattice path to a decreasing $k$-Naples parking function. Given a lattice path $L\in LP_{n,k}$, we know this path starts at $(0,n)$, ends at $(n,0)$, and stays below the line $y=n-x+k$. Suppose the east steps of $L$ occur from $(i-1,a_i-1)$ to $(i,a_i-1)$, then by definition $a_i\leq \min(n,n+k+1-i)$ for all $1\leq i\leq n$. Construct the parking preference $\alpha=(a_1,a_2,\dots,a_n)$. 
Note that the construction of $\alpha$ guarantees that $\alpha$ is in decreasing order.
It remains for us to show that $\alpha\in PF_{n,k}$. That is, we check that for all $i\in [n]$, car $c_i$ can park under the $k$-Naples parking rules.

First, observe that the first $k+1$ cars can always park, since they can back up to $k$ positions, see Equation \eqref{eq:pf=pp}. 
Now for $i>k+1$ we split into two cases:
    $a_i\leq k$ and 
$a_i>k$.\\

\noindent
\emph{Case 1:} Suppose $i>k+1$ and $a_i\leq k$. If one of the spots between $1$ and $k$ is unoccupied, then $c_i$ parks there. Instead, if all of the parking spots between 1 and $k$ are occupied, this means that of the $i-1$ cars that have parked, $k$ of them have occupied the first $k$ spots, while the remaining $i-(k+1)$ cars occupied some spots numbered between $k+1$ and $n$. Thus, there are less cars parked to the right of spot $k$, than there are parking spots between $k+1$ and $n$.  
Thus, $c_i$ parks in the leftmost available spot between $k+1$ and $n$.\\

\noindent
\emph{Case 2:} Suppose $i>k+1$ and $a_i>k$. If spot $a_i$ or a spot up to $k$ steps behind is unoccupied then $c_i$ can park. 
Otherwise, spots $a_i-k,a_{i}-k+1,\ldots , a_i$ are occupied. Now, since $a_i\leq n+k+1-i$, we have that there are  $n-a_i\geq n-(n+k+1-i)=i-(k+1)$ spots to the right of $a_i$. By assumption, spots $a_i-k$ through $a_i$ are occupied by $k+1$ of the $i - (k+1)$ cars that parked before $c_i$ so that $i-(k+2)$ cars have parked in the at least $i-(k+2)$ spots after $a_i$. This leaves a remaining open spot between $a_{i+1}$ and $n$ in which $c_i$ can park. 
Since $i$ was arbitrary, we have established that all cars $c_i$ can park for all $1\leq i\leq n$. Thus, $\alpha \in PF_{n,k}$. 
\end{proof}

Now we provide a connection between $k$-lattice paths and signature Dyck paths. In a recent paper by Cellabos and Gonz\'alez D'Le\'on  \cite{D'LeonandCeballos}, they introduce the concept of signature Dyck paths, defined by a vector $s=(s_1,s_2,\ldots,s_\ell)\in \N^\ell$. The signature $s$ defines a ribbon and an $s$-Dyck path is a lattice path that lies on or above the ribbon.
To describe the ribbon, we construct a grid of dimensions $\ell\times [(\sum_{i=1}^{\ell}s_i-1)+1]$. We number the boxes in this grid from bottom to top, calling each row a level $i$ with $1\leq i\leq \ell$, and on each level we number the boxes left to right from $1$ to $(\sum_{i=1}^{\ell}s_i-1)+1$.
At each level we shade a specific set of boxes. Begin by shading the boxes $1,2,\ldots,s_1$ at level $1$. Then, for $2\leq i\leq \ell$, shade the boxes numbered $(\sum_{j=1}^{i-1}s_{j}-1)+1$ to $(\sum_{j=1}^{i}s_{j}-1)+1$ at level $i$. Figure \ref{fig:exsd}, illustrates the ribbon when $s=(3,2,5,1,1)$, along with an $s$-Dyck path in blue, and a lattice path that is not an $s$-Dyck path in red.

\begin{figure}[h]
    \begin{subfigure}[b]{0.3\textwidth}
    \[
    \begin{tikzpicture}[scale=0.55]
		\draw[step=1cm] (0,0) grid (8,5);
		\foreach \i in {1,...,8}{
			\node at (-0.5+\i,-0.5) {$\i$};
		}
		\foreach \j in {1,...,5}{
			\node at (-0.5,-0.5+\j) {$\j$};
		}
		\draw[pattern=north west lines, pattern color=black] (0,0) rectangle (3,1);
		\draw[pattern=north west lines, pattern color=black] (2,1) rectangle 	(4,2);
		\draw[pattern=north west lines, pattern color=black] (3,2) rectangle (8,3);
		\draw[pattern=north west lines, pattern color=black] (7,3) rectangle (8,5);
	\end{tikzpicture}\]
	\caption{Ribbon corresponding to $s$}
	\end{subfigure}
    \hfill
    \begin{subfigure}[b]{0.3\textwidth}
    \[
    \begin{tikzpicture}[scale=0.55]
		\draw[step=1cm] (0,0) grid (8,5);
		\foreach \i in {1,...,8}{
			\node at (-0.5+\i,-0.5) {$\i$};
		}
		\foreach \j in {1,...,5}{
			\node at (-0.5,-0.5+\j) {$\j$};
		}
		\draw[pattern=north west lines, pattern color=black] (0,0) rectangle (3,1);
		\draw[pattern=north west lines, pattern color=black] (2,1) rectangle (4,2);
		\draw[pattern=north west lines, pattern color=black] (3,2) rectangle (8,3);
		\draw[pattern=north west lines, pattern color=black] (7,3) rectangle (8,5);
		\draw[line width=3pt, blue] (0,0) to (0,2) to(3,2) to (3,3) to (5,3) to (5,4) to (7,4) to (7,5) to (8,5);
	\end{tikzpicture}
    \]
    \caption{An $s$-Dyck path}
    \end{subfigure}
    \hfill
    \begin{subfigure}[b]{0.3\textwidth}
    \[\begin{tikzpicture}[scale=0.55]
		\draw[step=1cm] (0,0) grid (8,5);
		\foreach \i in {1,...,8}{
			\node at (-0.5+\i,-0.5) {$\i$};
		}
		\foreach \j in {1,...,5}{
			\node at (-0.5,-0.5+\j) {$\j$};
		}
		\draw[pattern=north west lines, pattern color=black] (0,0) rectangle (3,1);
		\draw[pattern=north west lines, pattern color=black] (2,1) rectangle (4,2);
		\draw[pattern=north west lines, pattern color=black] (3,2) rectangle (8,3);
		\draw[pattern=north west lines, pattern color=black] (7,3) rectangle (8,5);
		\draw[line width=3pt, red] (0,0) to (0,1) to(3,1) to (3,3) to (5,3) to (5,4) to (7,4) to (7,5) to (8,5);
	\end{tikzpicture}\]
    \caption{Not an $s$-Dyck path}
    \end{subfigure}
    \caption{The ribbon corresponding to the signature $s=(3,2,5,1,1)$, and two lattice paths: one an $s$-Dyck path (blue path) and that is not an $s$-Dyck path (red path).}
    \label{fig:exsd}
\end{figure}
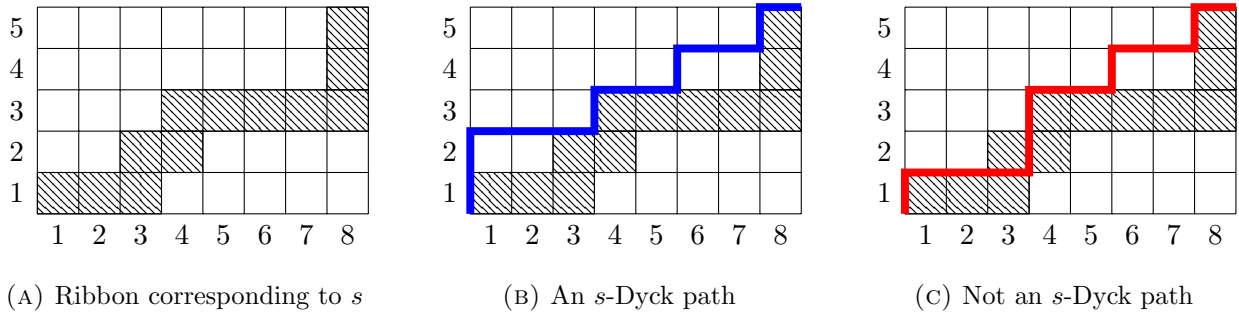

In our work, we consider a horizontal reflection of $s$-Dyck paths so that our paths are decreasing, rather than increasing. In this way, $k$-lattice paths of length $2n$ are $s$-Dyck paths with signature 
\begin{align}
    s = (k+1, \underbrace{2, 2, \dots, 2}_{n-k}, \underbrace{1, 1, \dots, 1}_{k})\label{eq:sigs}
\end{align}
of length $2(n+1).$ In Figure \ref{fig:exofasig}, we illustrate the signature for a $3$-lattice path of length $14$, and note that any lattice path begins with a south step from $(0,7)$ to $(0,6)$ and ends with an east step from $(6,0)$ to $(7,0)$.

\begin{figure}[h]
    \centering
	\[
	\begin{tikzpicture}[scale=0.6]
		\draw[step=1cm] (0,0) grid (7,7);
		\draw[pattern=north west lines, pattern color=black] (0,7) rectangle (4,6);
		\draw[pattern=north west lines, pattern color=black] (3,6) rectangle (5,5);
		\draw[pattern=north west lines, pattern color=black] (4,5) rectangle (6,4);
		\draw[pattern=north west lines, pattern color=black] (5,4) rectangle (7,3);
		\draw[pattern=north west lines, pattern color=black] (6,3) rectangle (7,0);
		\draw[red, line width=3pt] (1.5,7.5) to (7.5,1.5);
	\end{tikzpicture}
	\]  
    \caption{Illustrating the possible locations for $3$-lattice paths of length $12$, which begin at $(0,6)$, end at $(6,0)$, and must lie below the red line given by $y=6-x+3$. This corresponds to $s$-Dyck paths of length $14$, which begin at $(0,7)$, end at $(7,0)$, and must lie on or below the signature $s=(4,2,2,2,1,1,1)$.}
    \label{fig:exofasig}
\end{figure}
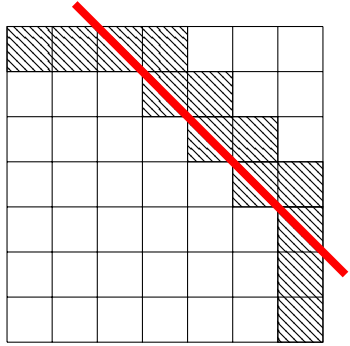

 \begin{table} 
    \begin{tabular}{|c|l|l|l|}\hline
     Subset of $k$-Naples parking functions & OEIS & Sequence & Formula\\
         \hline\hline
      $|PF_{n,1}^d|$ with $n\geq 2$ &\textcolor{blue}{\href{http://oeis.org/A000245}{A000245}}& $3, 9, 28, 90, 297,\ldots$& $ \frac{3(2n)!}{(n+2)!(n-1)!}$\\\hline
      $|PF_{n,2}^d|$ with $n\geq 3$ &\textcolor{blue}{\href{http://oeis.org/A026016}{A026016}}& $10, 34, 117, 407,\ldots$ & $\binom{2(n-1)}{n} - \binom{2(n-1)}{(n+3)}$\\\hline
      $|PF_{n,3}^d|$ with $n\geq 4$ &\textcolor{blue}{\href{http://oeis.org/A026026}{A026026}}& $35, 125, 451, 1638,\ldots$ & $ \binom{2n-1}{n-1}-\binom{2n-1}{n-5}$ \\\hline
    \end{tabular}
\caption{Known integer sequences related to enumerating decreasing $k$-Naples parking functions, which we denote as $PF_{n,k}^d$.
}\label{tab:specialcases}
\end{table}

Theorem \ref{DecreasingNaplesiffDyck} along with the resulting sequences in Table \ref{tab:specialcases} give formulas for the number of $s$-Dyck paths with signatures as given in Equation \eqref{eq:sigs}, for the special cases $k=1$, $2$, and $3$. We note that formulas for other values of $k$ are unknown.

\subsection{Rearrangements of $k$-Naples parking functions}
We begin by illustrating that not all rearrangements of $k$-Naples parking functions are $k$-Naples parking functions.

\begin{example}\label{k-naples example}
Let $\alpha = (7,7,7,7,5,2,2)\in PP_{7}$. We now verify that $\alpha$ is an element of $PF_{7,3}$. First, $c_1$ parks in its preferred parking spot $7$. Then, $c_2$ backs up one space and parks in position $6$, $c_3$ backs up two spaces and parks in position $5$, $c_4$ backs up three spaces and parks in position $4$, $c_5$ backs up two spaces and parks in position $3$. 
Next, $c_6$ parks in its preferred parking space $2$, while the last car, $c_7$, has to back up one space to park in position $1$. 
The filled parking lot based on $\alpha$ is illustrated in Figure \ref{fig:npf2}.

\begin{figure}[h]
    \centering
\begin{tikzpicture}
\draw[thick] (0,0) to node[below]{1} node[above]{$c_7$} (0.5,0);
\draw[thick] (1,0) to node[below]{2} node[above]{$c_6$}(1.5,0);
\draw[thick] (2,0) to node[below]{3} node[above]{$c_5$} (2.5,0); 
\draw[thick] (3,0) to node[below]{4} node[above]{$c_4$}(3.5,0);
\draw[thick] (4,0) to node[below]{5} node[above]{$c_3$}(4.5,0);
\draw[thick] (5,0) to node[below]{6} node[above]{$c_2$} (5.5,0);
\draw[thick] (6,0) to node[below]{7} node[above]{$c_1$}(6.5,0);
\end{tikzpicture}    
    \caption{Illustrating the parking order for the $3$-Naples parking function $(7,7,7,7,5,2,2)$.}
    \label{fig:npf2}
\end{figure}
Now, notice that in the rearrangement $\beta=(5,7,7,7,7,2,2)$ of $\alpha$, $c_1$ parks at position 5, $c_2$ parks at position 7, $c_3$ backs up one space and parks at position 6, and $c_4$ backs up three spaces and parks at position $4$. Now $c_5$ finds its preferred space and the three preceding occupied. Additionally, when it checks forward, there are no available spaces and $c_5$ cannot park. Thus, $\beta\notin PF_{7,3}$. 
\end{example}

Characterizing when a rearrangement of a $k$-Naples parking function is another $k$-Naples parking function remains an open problem.
We state this formally below.

\begin{problem} 
Characterize and enumerate which rearrangements of decreasing $k$-Naples parking functions are also elements of $PF_{n,k}$.
\end{problem}

In the case where $k=1$ we conjecture the following.

\begin{conjecture}\label{OneCornerAboveLine}
If there is only one corner above the line $y= n-x$ then that parking preference and all of its rearrangements are elements of $PF_{n,1}$.
\end{conjecture}

\section*{Acknowledgements}
The authors thank Alyson Baumgardner and Katie Johnson for introducing us to the Naples parking function problem. We also thank Ayomikun Adeniran for helpful conversations at the beginning stages of this project.
\bibliographystyle{alpha}
\bibliography{ref}

\begin{bibdiv}
\begin{biblist}

\bib{BaumgardnerHonorsContract}{techreport}{
      author={Baumgardner, Alyson},
       title={The naples parking function},
      series={Honors Contract-Graph Theory},
   publisher={Florida Gulf Coast University},
     address={Florida Gulf Coast University},
        date={2019},
}

\bib{nmParkingFunctions}{article}{
      author={Cameron, Peter~J.},
      author={Johannsen, Daniel},
      author={Prellberg, Thomas},
      author={Schweitzer, Pascal},
       title={Counting defective parking functions},
        date={2008},
        ISSN={1077-8926},
     journal={Electron. J. Combin.},
      volume={15},
      number={1},
       pages={Research Paper 92, 15},
         url={http://www.combinatorics.org/Volume_15/Abstracts/v15i1r92.html},
      review={\MR{2426155}},
}

\bib{D'LeonandCeballos}{article}{
      author={Ceballos, Cesar},
      author={Gonz\'alez~D'Le\'on, Rafael~S.},
       title={Signature catalan combinatorics},
        date={2019},
     journal={Journal of Combinatorics},
      volume={10},
       pages={725\ndash  773},
}

\bib{PollakCircleArgument}{article}{
      author={Foata, Dominique},
      author={Riordan, John},
       title={Mappings of acyclic and parking functions},
        date={1974},
        ISSN={0001-9054},
     journal={Aequationes Math.},
      volume={10},
       pages={10\ndash 22},
         url={https://doi.org/10.1007/BF01834776},
      review={\MR{0335294}},
}

\bib{KonheimAndWeiss}{article}{
      author={G.~Konheim, Alan},
      author={Weiss, Benjamin},
       title={An occupancy discipline and applications},
        date={196611},
     journal={Siam Journal on Applied Mathematics - SIAMAM},
      volume={14},
}

\bib{Pyke}{article}{
      author={Pyke, Ronald},
       title={The supremum and infimum of the poisson process},
        date={195906},
     journal={Ann. Math. Statist.},
      volume={30},
      number={2},
       pages={568\ndash 576},
         url={https://doi.org/10.1214/aoms/1177706269},
}

\bib{pfdt}{article}{
      author={Schumacher, Paul R.~F.},
       title={Descents in parking functions},
        date={2018},
        ISSN={1530-7638},
     journal={J. Integer Seq.},
      volume={21},
      number={2},
       pages={Art. 18.2.3, 8},
      review={\MR{3779772}},
}

\bib{xParkingFunctions}{article}{
      author={Yan, Catherine~H.},
       title={Generalized parking functions, tree inversions, and multicolored
  graphs},
        date={2001},
        ISSN={0196-8858},
     journal={Adv. in Appl. Math.},
      volume={27},
      number={2-3},
       pages={641\ndash 670},
         url={https://doi.org/10.1006/aama.2001.0754},
        note={Special issue in honor of Dominique Foata's 65th birthday
  (Philadelphia, PA, 2000)},
      review={\MR{1868985}},
}

\bib{YanSurvey}{incollection}{
      author={Yan, Catherine~H.},
       title={Parking functions},
        date={2015},
   booktitle={Handbook of enumerative combinatorics},
      editor={B\'{o}na, Mikl\'{o}s},
   publisher={CRC Press, Boca Raton, FL},
       pages={835\ndash 893},
}

\end{biblist}
\end{bibdiv}

\end{document}